\documentclass[reqno,12pt,a4paper]{amsart}
\usepackage{bbm}
\usepackage{}

\usepackage{tikz}
\usetikzlibrary{matrix,arrows,calc,snakes,patterns,decorations.markings}

\usepackage[mathscr]{eucal}
\usepackage{graphics,epic}
\usepackage{amsfonts}
\usepackage{amscd}
\usepackage{latexsym}
\usepackage{amsmath,amssymb, amsthm, stmaryrd, bm, bbm}
\usepackage[all,2cell]{xy}
\usepackage{mathrsfs}
\usepackage{color}

\newtheorem{theorem}[subsection]{Theorem}
\newtheorem*{theorem*}{Theorem}

\newtheorem{lemma}[subsection]{Lemma}
\newtheorem{proposition}[subsection]{Proposition}

\newtheorem*{conjecture*}{Conjecture}

\newtheorem*{question*}{Question}
\theoremstyle{remark}

\newtheorem{example}[subsection]{Example}
\theoremstyle{definition}
\newtheorem{definition}[subsection]{Definition}
\newtheorem*{notation*}{Notation}

\numberwithin{equation}{section}

\begin{document}

\title[]{Excellent extensions and homological conjectures}

\author{Yingying Zhang}

\address{Department of Mathematics, Hohai University, Nanjing 210098, Jiangsu Province, P.R. China}

\email{zhangying1221@sina.cn}

\begin{abstract}
In this paper, we introduce the notion of excellent extension of rings. Let $\Gamma$ be an excellent extension of an artin algebra $\Lambda$, we prove that $\Lambda$ satisfies the Gorenstein symmetry conjecture (resp. finitistic dimension conjecture, Auslander-Gorenstein conjecture, Nakayama conjecture) if and only if so does $\Gamma$.  As a special case of excellent extensions,  if $G$ is a finite group whose order is invertible in $\Lambda$ acting on $\Lambda$ and $\Lambda$ is $G$-stable,  we prove that if the skew group algebras $\Lambda G$ satisfies strong Nakayama conjecture (resp. generalized Nakayama conjecture), then so does $\Lambda$.
\end{abstract}

\maketitle

\section{Introduction}\label{s:introduction}
\medskip
Let $\Lambda$ be an artin algebra and all modules are finitely generated unless stated otherwise.
Denote by $\mathrm{mod}\Lambda$ the category of finitely generated left $\Lambda$-modules. For a module $M\in\mathrm{mod}\Lambda$, ${\rm pd}_{\Lambda}M$ and ${\rm id}_{\Lambda}M$
are the projective and injective dimensions of $M$ respectively.

The following homological conjectures are very important in the representation theory of artin algebras.

\textbf{Auslander-Reiten Conjecture (ARC)} Any module $M\in\mathrm{mod}\Lambda$ satisfying ${\rm Ext}_{\Lambda}^{\geq 1}(M, M\oplus \Lambda)=0$
implies that $M$ is projective.

\vspace{0.2cm}

\textbf{Gorenstein Projective Conjecture (GPC)} If $M$ is a Gorenstein projective $\Lambda$-module
such that ${\rm Ext}_{\Lambda}^{\geq 1}(M, M)=0$, then $M$ is projective.

\vspace{0.2cm}

\textbf{Strong Nakayama Conjecture (SNC)} Any module $M\in\mathrm{mod}\Lambda$ satisfying ${\rm Ext}_{\Lambda}^{\geq 0}(M, \Lambda)=0$
implies $M=0$.

\vspace{0.2cm}

\textbf{Generalized Nakayama Conjecture (GNC)} For any simple module $S\in\mathrm{mod}\Lambda$, there exists $i\geq 0$ such that
${\rm Ext}_{\Lambda}^{i}(S, \Lambda)\neq 0$.

\vspace{0.2cm}

Let
$$0\rightarrow\Lambda\rightarrow I^{0}\rightarrow I^{1}\rightarrow\cdots$$ be a minimal injective resolution of the $\Lambda$-module $\Lambda$.

\vspace{0.2cm}

\textbf{Auslander-Gorenstein Conjecture (AGC)}
If pd$_{\Lambda}I^{i}\leq i$ for any $i\geq 0$, then $\Lambda$ is Gorenstein (that is, the left and right self-injective dimensions of $\Lambda$ are finite).

\vspace{0.2cm}

\textbf{Nakayama Conjecture (NC)} If $I^{i}$ is projective for any $i\geq0$, then $\Lambda$ is self-injective.

\vspace{0.2cm}

\textbf{Wakamatsu Tilting Conjecture (WTC)} Let $T_{\Lambda}$ be a Wakamatsu tilting module with $\Gamma={\rm End}(T_{\Lambda})$.
Then ${\rm id\,}T_{\Lambda}={\rm id\,}_{\Gamma}T$.

\vspace{0.2cm}

\textbf{Gorenstein Symmetric Conjecture (GSC)} ${\rm id\,}\Lambda_{\Lambda}={\rm id\,}_{\Lambda}\Lambda$.

\vspace{0.2cm}

\textbf{Finitistic Dimension Conjecture (FDC)} ${\rm findim\,}\Lambda:={\rm sup\,}\{{\rm pd\,}_{\Lambda}M|M\in{\rm mod\,}\Lambda$
with ${\rm pd\,}_{\Lambda}M<\infty\}<\infty$.

\vspace{0.2cm}

These conjectures remain still open now, and there are close relations among them as follows which means the implications hold true for each artin algebra.
$$\xymatrix{
FDC \ar@{=>}[r] \ar@{=>}[d] & SNC \ar@{=>}[r] & GNC \ar@{=>}[r] \ar@{<=>}[d] & AGC \ar@{=>}[r] & NC \\
WTC \ar@{=>}[d] \ar@{=>}[urr] && ARC \ar@{=>}[d] && \\
GSC && GPC. &&
}$$
We refer to [AR, BFS, CT, FZ, LH, LJ, W1, W2, W3, X1, X2, X3, Y, Z] for details.

The notion of excellent extension of rings was introduced by Passman in [P] which is important in studying the algebraic structure of group rings.  We will give some common examples of excellent extension of rings in this paper (see Example 2.2 for details). Many algebraists have studied the invariant properties of artin algebras under excellent extensions such as the projectivity, injectivity, finite representation type, CM-finite, CM-free and representation dimension (see [Bo, HS, L, P, PS] and so on). As a special case of excellent extensions, Reiten and Riedtmann introduced the notion of skew group algebras in [RR].  In this paper, we will connect excellent extensions with homological conjectures. The outline of this article is as follows.

In Section 2, we give some terminology and some known results that will be used in the later part.

In Section 3, we aim to prove the following
\begin{theorem}{\rm ([Theorem 3.2 and 3.3])}
If $\Gamma$ is an excellent extension of an artin algebra $\Lambda$, then $\Lambda$ satisfies AGC (resp. NC, GSC, FDC) if and only if so does $\Gamma$.
\end{theorem}

\begin{theorem}{\rm ([Theorem 3.6])}
Let $\Lambda$ be an artin algebra and $G$ a finite group whose order is invertible in $\Lambda$ acting on $\Lambda$. If $\Lambda$ is $G$-stable and the skew group algebra $\Lambda G$ satisfies SNC
(resp. GNC), then so does $\Lambda$.
\end{theorem}

\section{Preliminaries}

\vspace{0.2cm}

In this section, we give some terminology and preliminary results.

\subsection*{Excellent extensions}

First we recall the notion of weak excellent extensions of rings as a generalization of that of excellent extensions of rings.
\begin{definition}
Let $\Lambda$ be a subring of a ring $\Gamma$ such that $\Lambda$ and $\Gamma$ have the same identity. Then $\Gamma$ is called a ring extension of $\Lambda$, and denoted by $\Gamma\geq\Lambda$. A ring extension $\Gamma\geq\Lambda$ is called a weak excellent extension if:
\begin{itemize}
\item[(1)]
$\Gamma$ is right $\Lambda$-projective [P, p. 273], that is $N_{\Gamma}$ is a submodule of $M_{\Gamma}$ and if $N_{\Lambda}$ is a direct summand of $M_{\Lambda}$, denote by $N_{\Lambda}|M_{\Lambda}$, then $N_{\Gamma}|M_{\Gamma}$.
\item[(2)]
$\Gamma$ is finite extension of $\Lambda$, that is, there exist $\gamma_{1},\ldots, \gamma_{n}\in\Gamma$ such that $\Gamma=\sum^{n}_{i=1}\gamma_{i}\Lambda$.
\item[(3)]
$\Gamma_{\Lambda}$ is flat and $_{\Lambda}\Gamma$ is projective.
\end{itemize}
\end{definition}

Recall from [Bo, P] that a ring extension $\Gamma\geq\Lambda$ is called an excellent extension if it is weak excellent and $\Gamma_{\Lambda}$ and $_{\Lambda}\Gamma$ are free with a common basis ${\gamma_{1},\ldots, \gamma_{n}}$, such that $\Lambda\gamma_{i}=\gamma_{i}\Lambda$ for any $1\leq i\leq n$.  Here we list some examples of excellent extensions.

\begin{example}{\rm [ARS, Bo, P, RR]}
\begin{itemize}
\item[(1)]
For a ring $\Lambda$, $M_{n}(\Lambda)$ (the matrix ring of $\Lambda$ of degree $n$) is an excellent extension of $\Lambda$.
\item[(2)]
Let $\Lambda$ be a ring and $G$ a finite group. If $|G|^{-1}\in\Lambda$, then the skew group ring $\Lambda G$ is an excellent extension of $\Lambda$.
\item[(3)]
Let $A$ be a finite-dimensional algebra over a field $K$, and let $F$ be a finite separable field extension of $K$. Then $A\otimes_{K}F$ is an excellent extension of $A$.
\item[(4)]
Let $K$ be a field, and let $G$ be a group and $H$ a normal subgroup of $G$. If $[G:H]$ is finite and is not zero in $K$, then $KG$ is an excellent extension of $KH$.
\item[(5)]
Let $K$ be a field of characteristic $p$, and let $G$ be a finite group and $H$ a normal subgroup of $G$. If $H$ contains a Sylow $p$-subgroup of $G$, then $KG$ is an excellent extension of $KH$.
\item[(6)]
Let $K$ be a field and $G$ a finite group. If $G$ acts on $K$ (as field automorphisms) with kernel $H$, then the skew group ring $K\ast G$ is an excellent extension of the group ring $KH$, and the center $Z(K\ast G)$ of $K\ast G$ is an excellent extension of the center $Z(KH)$ of $KH$.
\end{itemize}
\end{example}

\begin{proposition}{\rm [HS, Lemma 3.5]}
Let $\Gamma\geq \Lambda$ be a weak excellent extension. If $\Lambda$ is an artin algebra, then so is $\Gamma$.
\end{proposition}

From now on, let $\Lambda$ be an artin algebra and $\Gamma\geq\Lambda$ be an excellent extension. Then by Proposition 2.3 it follows that $\Gamma$ is also an artin algebra. By the adjoint isomorphism theorem we have the following adjoint pair ($F$,$H$): $$F:=\Gamma_{\Lambda}\otimes-: {\rm mod\,}\Lambda\to {\rm mod\,}\Gamma,$$
$$H:={\rm Hom}_{\Gamma}(\Gamma,-): {\rm mod\,}\Gamma\to {\rm mod\,}\Lambda.$$
Then by [HS, Lemma 4.7] we have
\begin{lemma}
Both $(F,H)$ and $(H,F)$ are adjoint pairs.
\end{lemma}
So it follows that $F$ and $H$ are both exact functors therefore preserve projective and injective modules.

\subsection*{Skew group algebras}
Let $\Lambda$ be an algebra and $G$ be a group with identity 1 acting on $\Lambda$, that is, a map $G\times \Lambda \longrightarrow \Lambda$ via $(\sigma,\lambda)\mapsto \sigma(\lambda)$ such that
\begin{itemize}
\item[(a)] The map $\sigma$: $\Lambda \rightarrow \Lambda$ is an algebra automorphism for each $\sigma$ in $G$.
\item[(b)] $(\sigma_{1} \sigma_{2})(\lambda)=\sigma_{1}(\sigma_{2}(\lambda))$ for all $\sigma_{1}$,$\sigma_{2}\in G$ and $\lambda\in\Lambda$.
\item[(c)] 1$(\lambda)$=$\lambda$ for all $\lambda \in\Lambda$.
\end{itemize}

Let $\Lambda$ be an artin algebra, $G$ a finite group whose order is invertible in $\Lambda$ and
$G\longrightarrow Aut\Lambda$ a group homomorphism. The data involved in defining a new category of $\Lambda$-modules in terms of $G$:
\begin{itemize}
\item[(1)] For any $X\in \mathrm{mod}\Lambda$ and $\sigma\in G$,
let $^{\sigma}\hspace{-2pt}X$ be the $\Lambda$-module as follows: as a $k$-vector space $^{\sigma}\hspace{-2pt}X=X$,
the action on $^{\sigma}\hspace{-2pt}X$ is given by $\lambda\cdot x=\sigma^{-1}(\lambda)x$ for all $\lambda\in\Lambda$ and $x\in X$.
\item[(2)] Given a morphism of $\Lambda$-modules $f:X\longrightarrow Y$, define $^{\sigma}\hspace{-2pt}f$ : $^{\sigma}\hspace{-2pt}X$
$\longrightarrow$ $^{\sigma}\hspace{-2pt}Y$ by $^{\sigma}\hspace{-2pt}f(x)=f(x)$ for each $x\in$ $^{\sigma}\hspace{-2pt}X$.
\end{itemize}
Then $^{\sigma}\hspace{-2pt}f$ is also a $\Lambda$-homomorphism. Indeed, for $x\in X$ and $\lambda\in\Lambda$ we have
$^{\sigma}\hspace{-2pt}f(\lambda\cdot x)=f(\sigma^{-1}(\lambda)x)=\sigma^{-1}(\lambda)f(x)=\lambda\cdot$ $^{\sigma}\hspace{-2pt}f(x)$.
Using the above setup, we can define a functor $F_{\sigma}$
by $F_{\sigma}(X)$=$^{\sigma}\hspace{-2pt}X$ and $F_{\sigma}(f)$=$^{\sigma}\hspace{-2pt}f$ for $X,Y\in \mathrm{mod}\Lambda$ and homomorphism $f:X\longrightarrow Y$.

\vspace{0.2cm}

Then one can get the following observation immediately.
\begin{proposition}
$F_{\sigma}:\mathrm{mod}\Lambda\longrightarrow \mathrm{mod}\Lambda$ is an automorphism and the inverse is $F_{\sigma^{-1}}$.
\end{proposition}
To state our main results in this paper, we need the following definition from [RR].

\begin{definition}
The {\it skew group algebra} $\Lambda G$ that $G$ acts on $\Lambda$ is given by the following data:
\begin{itemize}
\item[(a)] As an abelian group, $\Lambda G$ is the free left $\Lambda$-module with the elements of $G$ as a basis.

\item[(b)] For all $\lambda_{\sigma}$ and $\lambda_{\tau}$ in $\Lambda$
and $\sigma$ and $\tau$ in $G$, the multiplication in $\Lambda G$ is defined by the rule $(\lambda_{\sigma}\sigma)(\lambda_{\tau}\tau)
=(\lambda_{\sigma}\sigma(\lambda_{\tau}))\sigma\tau$.
\end{itemize}
\end{definition}

\vspace{0.2cm}

In particular, when $G$ is a finite group whose order is invertible in $\Lambda$, the natural inclusion $\Lambda\hookrightarrow\Lambda G$ induces the restriction functor $H: \mathrm{mod}\Lambda G\longrightarrow \mathrm{mod}\Lambda$ and the induction functor $F: \mathrm{mod}\Lambda\longrightarrow \mathrm{mod}\Lambda G$ which are the same as above when $\Gamma=\Lambda G$. We recall the following facts from [RR, pp.227, 235].
\begin{proposition}

\begin{itemize}
\item[]
\item[(a)] Let $M\in \mathrm{mod}\Lambda$ and $\sigma\in G$. We have isomorphisms of
$FM\cong \bigoplus_{\sigma\in G}(\sigma\otimes M)\cong\bigoplus_{\sigma\in G}$$^{\sigma}\hspace{-2pt}M$ as
$\Lambda$-modules. Then $HFM\cong \bigoplus_{\sigma\in G}(\sigma\otimes M)\cong\bigoplus_{\sigma\in G}$$^{\sigma}\hspace{-2pt}M$.
\item[(b)] The natural morphism $I\rightarrow HF$ is a split monomorphism of functions, where
$I: \mathrm{mod}\Lambda\rightarrow \mathrm{mod}\Lambda$ is the identity functor. Dually, the natural morphism
$FH\rightarrow J$ is a split epimorphism of functions, where $J: \mathrm{mod}\Lambda G\rightarrow \mathrm{mod}\Lambda G$
is the identity functor.
\end{itemize}
\end{proposition}
For the convenience of the readers, we give an easy example to understand skew group algebras. We refer to [RR] for more information.

\begin{example}
Let $\Lambda$ be the path algebra of the quiver Q. The cyclic group $G=\mathbb{Z}/2\mathbb{Z}$ acts on $\Lambda$ by switching
$2$ and $2'$, $\alpha$ and $\beta$ and fixing the vertex $1$.
Then the quivers of $\Lambda$ and $\Lambda G$ are as follows:
$$
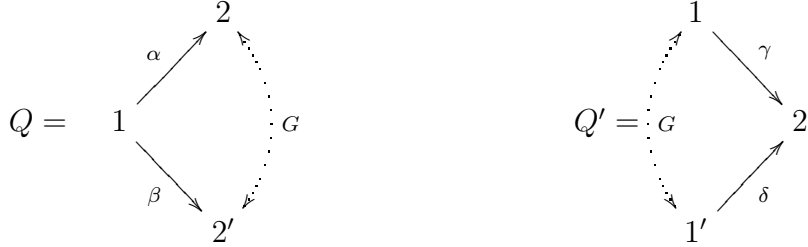
\xymatrix{
&&&2\ar@/^1.5pc/@{<.>}[dd]^{G} &&&&&1\ar@/_1.5pc/@{<.>}[dd]^{G}\ar[dr]^{\gamma} &&&\\
&Q=\hspace{-15pt}&1\ar[ur]^{\alpha}\ar[dr]_{\beta}&& &&&Q'=\hspace{-15pt}&&2&\\
&&&2'&&&&&1'\ar[ur]_{\delta}&&&
}$$
\end{example}

\section{Homological conjectures}

\vspace{0.2cm}

Let $\Lambda$ be an artin algebra and $\Gamma$ be its excellent extension. Since $F$ and $H$ preserve injective modules, we have the following result which states that $\Lambda$ and $\Lambda G$ have
the same self-injective dimension.

\begin{lemma}
${\rm id\,}\Lambda{_{\Lambda}}={\rm id\,}\Gamma{_{\Gamma}}$ and ${\rm id\,}{_{\Lambda}}\Lambda={\rm id\,}{_{\Gamma}}\Gamma.$
In particular, $\Lambda$ is self-injective(resp. Gorenstein) if and only if $\Gamma$ is self-injective(resp. Gorenstein).
\end{lemma}
\begin{proof}
Since $F$ and $H$ preserve injective modules, it follows that
$${\rm id\,}\Gamma{_{\Gamma}}={\rm id\,}F\Lambda{_{\Lambda}}\leq{\rm id\,}\Lambda{_{\Lambda}}\leq{\rm id\,}HF\Lambda{_{\Lambda}}
={\rm id\,}H\Gamma{_{\Gamma}}\leq{\rm id\,}\Gamma{_{\Gamma}}.$$ Thus ${\rm id\,}\Lambda{_{\Lambda}}={\rm id\,}\Gamma{_{\Gamma}}$.
Similarly, we have ${\rm id\,}{_{\Lambda}}\Lambda={\rm id\,}{_{\Gamma}}\Gamma.$
\end{proof}

\begin{theorem}
\begin{itemize}
\item[]
\item[(1)] $\Lambda$ satisfies GSC if and only if so does $\Gamma$.
\item[(2)] $\Lambda$ satisfies FDC if and only if so does $\Gamma$.
\end{itemize}
\end{theorem}

\begin{proof}
(1) By Lemma 3.1, it follows that $\Lambda$ satisfies GSC if and only if ${\rm id\,}\Lambda{_{\Lambda}}={\rm id\,}{_{\Lambda}}\Lambda$
if and only if ${\rm id\,}\Gamma{_{\Gamma}}={\rm id\,}{_{\Gamma}}\Gamma$ if and only if $\Gamma$ satisfies GSC.

(2) By [HS, Proposition 3.6(1)], we get the assertion.
\end{proof}

\begin{theorem}
\begin{itemize}
\item[]
\item[(1)] $\Lambda$ satisfies AGC if and only if so does $\Gamma$.
\item[(2)] $\Lambda$ satisfies NC if and only if so does $\Gamma$.
\end{itemize}
\end{theorem}
\begin{proof}
(1) Assume that $\Gamma$ satisfies AGC. Let
\begin{equation}
0\rightarrow\Lambda\rightarrow I^{0}\rightarrow I^{1}\rightarrow\cdots
\end{equation}
be a minimal injective resolution of the $\Lambda$-module $\Lambda$ with pd $I^{i}\leq i$ for any $i\geq 0$.
Applying the functor $F$ to (3.1) we have an injective resolution of $\Gamma$ as a $\Gamma$-module:
$$0\rightarrow\Gamma\rightarrow FI^{0}\rightarrow FI^{1}\rightarrow\cdots.$$
Take a minimal injective resolution of $\Gamma$ as a $\Gamma$-module: $$0\rightarrow\Gamma\rightarrow J^{0}\rightarrow J^{1}\rightarrow\cdots.$$
Then we have that $J^{i}$ is a direct summand of $FI^{i}$ for any $i\geq 0$. Therefore
${\rm pd\,}J^{i}\leq {\rm pd\,}FI^{i}\leq {\rm pd\,}I^{i}\leq i$ for any $i\geq0$. Then $\Gamma$ is
Gorenstein since $\Gamma$ satisfies AGC. By Lemma 3.1 we know that $\Lambda$ is Gorenstein.

Conversely, assume $\Lambda$ satisfies AGC. Take a minimal injective resolution of $\Gamma$ as a $\Gamma$-module:
\begin{equation}
0\rightarrow\Gamma\rightarrow J^{0}\rightarrow J^{1}\rightarrow\cdots
\end{equation}
with pd $J^{i}\leq i$ for any $i\geq 0$.
Applying the functor $H$ to (3.2) we get an injective resolution of $H\Gamma$ as a $\Lambda$-module:
$$0\rightarrow H\Gamma \rightarrow HJ^{0}\rightarrow HJ^{1}\rightarrow\cdots.$$
If $$0\rightarrow\Lambda\rightarrow I^{0}\rightarrow I^{1}\rightarrow\cdots$$ is a minimal injective resolution of $\Lambda$-module $\Lambda$,
then $I^{i}$ is a direct summand of $HJ^{i}$ for any $i\geq 0$ since $H\Gamma={_{\Lambda}\Gamma}$ is free. It follows that ${\rm pd\,}I^{i}\leq {\rm pd\,}HJ^{i}\leq {\rm pd\,}J^{i}\leq i$
for any $i\geq0$. Then $\Lambda$ is Gorenstein since $\Lambda$ satisfies AGC. By Lemma 3.1 we have that $\Gamma$ is Gorenstein.

\vspace{0.2cm}

(2) Assume that $\Gamma$ satisfies NC. Let
\begin{equation}
0\rightarrow\Lambda\rightarrow I^{0}\rightarrow I^{1}\rightarrow\cdots
\end{equation}
be a minimal injective resolution of the $\Lambda$-module $\Lambda$ with $I^{i}$ is projective for any $i\geq 0$.
Applying the functor $F$ to (3.3) we have an injective resolution of $\Gamma$ as a $\Gamma$-module:
$$0\rightarrow\Gamma\rightarrow FI^{0}\rightarrow FI^{1}\rightarrow\cdots.$$
Take a minimal injective resolution of $\Gamma$ as a $\Gamma$-module:
$$0\rightarrow\Gamma\rightarrow J^{0}\rightarrow J^{1}\rightarrow\cdots.$$
Then we have that $J^{i}$ is a direct summand of $FI^{i}$ for any $i\geq 0$. Therefore $J^{i}$ is projective
for any $i\geq 0$. Thus $\Gamma$ is self-injective since $\Gamma$ satisfies NC. By Lemma 3.1 we have that $\Lambda$ is self-injective.

Conversely, assume $\Lambda$ satisfies NC. Take a minimal injective resolution of $\Gamma$ as a $\Gamma$-module:
\begin{equation}
0\rightarrow\Gamma\rightarrow J^{0}\rightarrow J^{1}\rightarrow\cdots
\end{equation}
with $J^{i}$ is projective for any $i\geq 0$.
Applying the functor $H$ to (3.4) we get an injective resolution of $H\Gamma$ as a $\Lambda$-module:
$$0\rightarrow H\Gamma\rightarrow HJ^{0}\rightarrow HJ^{1}\rightarrow\cdots.$$
If $$0\rightarrow\Lambda\rightarrow I^{0}\rightarrow I^{1}\rightarrow\cdots$$ is a minimal injective resolution of $\Lambda$-module $\Lambda$,
then $I^{i}$ is a direct summand of $HJ^{i}$ for any $i\geq 0$ since $H\Gamma={_{\Lambda}\Gamma}$ is free. It follows that $I^{i}$ is projective for any $i\geq0$.
Then $\Lambda$ is self-injective since $\Lambda$ satisfies NC. By Lemma 3.1 we have that $\Gamma$ is self-injective.
\end{proof}

In particular, let $\Lambda$ be an artin algebra and $G$ be a finite group whose order $n$ is invertible in $\Lambda$ acting on $\Lambda$, now we connect skew group algebras with homological conjectures. Before doing this, we introduce the notion of $G$-stable.
\begin{definition}
A $\Lambda$-module $X$ is called $G$-stable if  ${^{\sigma}X}\cong X$ for any $\sigma\in G$. $\Lambda$ is called $G$-stable if it is $G$-stable as a left $\Lambda$-module.
\end{definition}

The following result plays a crucial role in the sequel.

\begin{proposition}
Let $M, N\in{\rm mod\,}\Lambda$ satisfying ${\rm Ext}_{\Lambda}^{i}(M, N)=0$ with $i\geq 0$. If $N$ is $G$-stable, then ${\rm Ext}_{\Lambda G}^{i}(FM, FN)=0$.
\end{proposition}
\begin{proof}
If $i=0$, then from the adjoint isomorphism theorem and Proposition 2.7(a) it follows that
$${\rm Hom}_{\Lambda G}(FM,FN)\cong {\rm Hom}_{\Lambda}(M,HFN)\cong\bigoplus_{\sigma\in G}{\rm Hom}_{\Lambda}(M,{^{\sigma}N})\cong ({\rm Hom}_{\Lambda}(M, N))^{n}.$$
We have finished to prove that ${\rm Hom}_{\Lambda G}(FM,FN)=0$.

If $i\geq 1$,
taking a projective resolution of $M$ in $\mathrm{mod}\Lambda$:
\begin{equation}
\cdots\rightarrow P_{1}\rightarrow P_{0}\rightarrow M\rightarrow 0
\end{equation}
then we have a projective resolution of $FM$ by applying the functor $F$: $$\cdots\rightarrow FP_{1}\rightarrow FP_{0}\rightarrow FM\rightarrow 0.$$
Set $P^{\bullet}=(\cdots\rightarrow P_{1}\rightarrow P_{0}\rightarrow0)$. Then $FP^{\bullet}=(\cdots\rightarrow FP_{1}\rightarrow FP_{0}\rightarrow0)$.
It follows that $${\rm Ext}_{\Lambda}^{i}(M, N)=H^{i}({\rm Hom}_{\Lambda}(P^{\bullet}, N))$$ is $i$th-homology of the complex ${\rm Hom}_{\Lambda}(P^{\bullet}, N)$
and $${\rm Ext}_{\Lambda G}^{i}(FM, FN)=H^{i}({\rm Hom}_{\Lambda G}(FP^{\bullet}, FN))$$ is $i$th-homology of the complex ${\rm Hom}_{\Lambda G}(FP^{\bullet}, FN)$.
By the adjoint isomorphism theorem and Proposition 2.7(a) we have ${\rm Hom}_{\Lambda G}(FP^{\bullet}, FN)\cong\bigoplus_{\sigma\in G}{\rm Hom}_{\Lambda}(P^{\bullet}, {^{\sigma}}N)\cong ({\rm Hom}_{\Lambda}(P^{\bullet}, N))^{n}$.
It follows that ${\rm Ext}_{\Lambda G}^{i}(FM, FN)=0$.
\end{proof}

\vspace{0.2cm}

In the following we give a connection between SNC and GNC for $\Lambda$ and that for $\Lambda G$.

\begin{theorem}
If $\Lambda$ is $G$-stable and $\Lambda G$ satisfies SNC (resp. GNC), then so does $\Lambda$.
\end{theorem}

\begin{proof}
(1) Assume $\Lambda G$ satisfies SNC. Let $M\in{\rm mod\,}\Lambda$ with ${\rm Ext}_{\Lambda}^{\geq 0}(M, \Lambda)=0$. By Proposition 3.5
we have ${\rm Ext}_{\Lambda G}^{\geq 0}(FM, \Lambda G)=0$. Then $FM=0$ since $\Lambda G$ satisfies SNC. It follows from Proposition 2.7(b)
that $M=0$ as a direct summand of $HFM=0$.

(2) Assume $\Lambda G$ satisfies GNC. Let $S\in{\rm mod\,}\Lambda$ be a simple module with ${\rm Ext}_{\Lambda}^{\geq 0}(S, \Lambda)=0$.
By Proposition 3.5 we have ${\rm Ext}_{\Lambda G}^{\geq 0}(FS, \Lambda G)=0$. By Propositions 2.7(a) and Proposition 2.5, we have that
$HFS\cong \bigoplus_{\sigma\in G}{^{\sigma}S}$ is a semisimple $\Lambda$-module. From [FJ, Theorem 4] we know that $FS$ is a semisimple
$\Lambda G$-module. Set $FS:=\bigoplus_{j\in J'}S_{j}'$, where $S_{j}'$ is simple for any $j\in J'$. Then
${\rm Ext}_{\Lambda G}^{\geq 0}(S_{j}', \Lambda G)=0$ for any $j\in J'$. Since $\Lambda G$ satisfies GNC, $S_{j}'=0$ for any $j\in J'$.
So $FS=0$. By Proposition 2.7(b), we have $S=0$ as a direct summand of $HFS=0$.
\end{proof}

We end this article with the following interesting question:

If $\Gamma$ is an excellent extension of an artin algebra $\Lambda$, does $\Lambda$ satisfy WTC (resp. SNC, GNC, ARC, GPC) if and only if so does $\Gamma$?

\vspace{0.5cm}

{\bf Acknowledgement.}
The author would like to thank Prof. Osamu Iyama for his hospitality during her stay in Nagoya with the support of CSC Fellowship.
Her special thanks are due to Prof. Zhaoyong Huang and Xiaojin Zhang for helpful
conversations on the subject. She also thanks the referee for the useful suggestions. This work was partially supported by NSFC (Grant No. 11571164), the program B for Outstanding PhD candidate of Nanjing University(201602B043) and also supported by the Fundamental Research Funds for the Central Universities(2017B07314).

\end{document}